\documentclass[leqno,final]{siamltex}
\usepackage{graphicx}%
\usepackage{multirow}%
\usepackage{amsmath,amstext,amssymb,bm}
\usepackage{mathrsfs}%
\usepackage[title]{appendix}%
\usepackage{xcolor}%
\usepackage{textcomp}%
\usepackage{manyfoot}%
\usepackage{booktabs}%
\usepackage{algorithm}%
\usepackage{algorithmicx}%
\usepackage{algpseudocode}%
\usepackage{listings}%

\usepackage{mathtools}
\usepackage{esint}

\usepackage{url} 
\usepackage{tikz}

 \numberwithin{equation}{section}
 \newtheorem{remark}{Remark}[section]

 \newtheorem{assumption}{Assumption}[section]
 \allowdisplaybreaks[4]

\def\al{\alpha}

\def\be{\beta}

\def\Ga{\Gamma}

\def\Om{\Omega}
\def\pa{\partial}

\def\sig{\sigma}

\def\R{\mathbb{R}}

\def\fa{\forall}
\def\grad{\bigtriangledown}

\def\Alg{\mathcal{A}}
\def\Bor{\mathcal{B}}

\def\Triag{\mathscr{T}}

\def\Alg{\mathcal{A}}
\def\Bor{\mathcal{B}}

\def\Triag{\mathscr{T}}

\def\argmin{\text{argmin}}

\def\diam{\text{diam}}

\def\dist{\text{dist}}

\def\sgn{\text{sgn}}

\begin{document}
	
		\title{$\Gamma$-convergence of an Enhanced Finite Element Method for Mani\`a's and Foss's Problems Exhibiting  the Lavrentiev Gap Phenomenon}
	    \markboth{XIAOBING FENG AND JOSHUA M. SIKTAR}{An Enhanced Finite Element Method for Lavrentiev Gaps}

	\author{
		Xiaobing Feng\thanks{Department of Mathematics, The University of Tennessee, Knoxville, TN 37996, U.S.A. (xfeng@utk.edu). } 
		\and 
	Joshua M. Sitkar\thanks{Department of Mathematics, The University of Tennessee, Knoxville, TN 37996, U.S.A. {\it Current address:} Department of Mathematics, Texas A \& M University, College Station, TX 77843, U.S.A.} }
	
	\date{ }
	
 \maketitle

	\begin{abstract} 
		 	It is well-known that numerically approximating calculus of variations problems possessing a Lavrentiev Gap Phenomenon (LGP) is challenging, and the standard numerical methodologies, such as finite element, finite difference, and discontinuous Galerkin methods, fail to give convergent methods because they cannot overcome the gap.  This paper is a continuation of \cite{feng2016enhanced}, where a promising enhanced finite element method was proposed to overcome the LGP in the classical Mani\`a's problem.  The first goal of this paper is to provide a complete $\Gamma$-convergence proof for this enhanced finite element method, hence, establishing a theoretical foundation for the method. The crux of the convergence analysis is taking advantage of the regularity of the minimizer and viewing the minimization problem as posed over the fractional Sobolev space  $W^{1 + s, p}(0, 1)$ (for $s > 0$) rather than the original admissible space $W^{1, p}(0, 1)$. The second goal is to extend the enhanced finite element method to the 
		 	two-dimensional Foss's problem  \cite{foss2003lavrentiev} from nonlinear elasticity, which is also known to possess the LGP, and to establish its $\Gamma$-convergence as well. 
	\end{abstract}
	
	
	

\section{Introduction}\label{Sec: intro}

The Mani\`a's problem refers to the following calculus of variations problem:
\begin{equation}\label{Eq: Mania}
    u  = \underset{v\in \Alg }{\argmin} \bigg\{J(v)  := \int^{1}_{0}  v'(x)^6 \bigl( v(x)^3-x \bigr)^2\, dx\bigg\}, 
\end{equation}
where $\Alg := \{u \in W^{1, 1}(0, 1) \ | \ u(0) = 0, \ u(1) = 1, \ J(u)<\infty \}$ (see Section \ref{Sec: notationAndPreliminaries} for notation).  
This problem is infamous for being one of the simplest that possesses the {\em Lavrentiev Gap Phenomenon} (LGP) in the sense that the following strict inequality holds (see \cite{mania1934sopra} or, more recently, \cite[Example 2.14]{Rin} for a proof):
\begin{equation}\label{Lav}
    \min_{u \in \Alg}J(u) \ < \ \inf_{u \in \Alg \cap W^{1, \infty}(0, 1)}J(u).
\end{equation}
Noticing that $W^{1, \infty}(0, 1)$ is a dense subset of $W^{1, 1}(0, 1)$,  hence, the LGP says that the same functional can have different minimum values (hence, different minimizers) over two nested admissible sets. 
	It is not hard to show that the minimizer of \eqref{Eq: Mania} is $\overline{u}(x) := x^{\frac{1}{3}}$, for which $J(\overline{u}) = 0$ (cf. \cite{ball1987one}).  
		It is easy to check that  $\overline{u}\in W^{1 + s, p}(0, 1)$ for $0\leq s<1/3$ and $1\leq p< 3/2$ when $(2/3 + s)p < 1$. 
		Consequently,  the minimizer of $J$ over $X:=\Alg \cap W^{1 + s, p}(0, 1)$ 
		is also $\overline{u}$ for the same range of $s$ and $p$.

The LGP is a difficult issue to address in the calculus of variations because it can occur not only for non-convex energy functionals but also for convex ones. Moreover, no general characterization of the LGP is known in the literature; consequently, it is often studied on a case-by-case basis. The LGP is of great interest in analysis, as it provides difficult and intriguing problems to study, in addition to having important applications in materials science.  It is also of great interest in numerical analysis because developing robust and convergent numerical methods for such problems has practical significance in providing means to compute their solutions. However, it is known that achieving such a goal is not only difficult but also intriguing. For example, it is well known that the standard finite element method fails to give convergence to a minimizer: that is, if $X_h$ denotes a  conforming linear  finite element space with respect to the uniform mesh over $(0,1)$ with mesh size $h>0$ that matches the boundary conditions of $\Alg$, then
\begin{equation}\label{noApprox}
    \lim_{h \rightarrow 0^+}\min_{u \in X_h}J(u_h) \ \nrightarrow  \ \min_{u \in \Alg}J(u)
\end{equation}
because $X_h\subset \Alg \cap W^{1, \infty}(0, 1)$.  We also refer the reader to  \cite{carstensen2009computation} for a discussion of this issue for general conforming finite element methods, and \cite{ortner2011nonconforming} for an analogous discussion regarding non-conforming finite element methods. An important message from this discussion is that if we want to use a conforming finite element method to overcome the LGP, we cannot use the original functional $J$ as an approximate functional to design a convergent numerical method. 

It is also well-understood that the cause of \eqref{noApprox} boils down to the fact that the factor $|u_h'|^6$ in the density function grows large too quickly as $h \rightarrow 0^+$ approaches the left-endpoint $x=0$ compared to the rate at which the other factor goes to zero. The papers \cite{ball1987numerical, li1995numerical, bai2006truncation} utilized a truncation method to approximate the singular minimizer. Another interesting numerical technique is the element removal method \cite{li1992element, li1995element}, which successfully approximated the true minimizer when given a priori information on the location of singularities in that minimizer.
 There were also additional papers surveying the analysis of the LGP in nonlinear elasticity problems \cite{almi2024new, foss2003examples}, in cavitation \cite{ball1987one, ball2002some}, and in microstructure theory \cite{winter1996lavrentiev}. More recent works \cite{balci2020new, balci2021lavrentiev, balci2022crouzeix}, meanwhile, have analyzed the LGP for nonlocal energy functionals.
In another more recent dissection of \eqref{Lav},  Feng and Schnake proposed in \cite{feng2016enhanced, schnake2017numerical} an enhanced finite element method using a cutoff technique. Their idea was to introduce an approximate functional that has a mechanism to limit the pointwise values of the derivative of the input function while still approximating the functional $J$ in some sense. Specifically, for some $\al > 0$ and $h > 0$, they introduced a cut-off function defined as
\begin{equation}\label{cutoff}
    \chi^{\al}_h(t) \ := \ \sgn(t)\min\{|t|, h^{-\al}\}, 
\end{equation}
for scalar-valued functions, and
\begin{equation}\label{Eq: cutoffVec}
    (\chi^{\al}_h(w_i)) \ := \ \sgn(w_i)\min\{|w_i|, h^{-\al}\},  \qquad  i=1,2,\cdots, d, 
\end{equation}
component-wise when $w$ is a $d$-vector.
Meanwhile, the corresponding approximate functional was defined as
\begin{equation}\label{cutoffMania}
    J^{\al}_h(v) \ := \ \int^{1}_{0} \bigl(\chi^{\al}_h(v'(x)) \bigr)^6 (v(x)^3 - x)^2\,dx,
\end{equation}
which led to the following enhanced finite element method: 
\begin{equation}\label{enhanced}
	\overline{u}_h \ = \ \underset{v_h\in X_h }{\argmin} \, J^{\alpha}_h(v_h). 
\end{equation}
We note that choices for the parameter $\alpha$ depend only on the structural properties of the density function of Mani\`a's functional and the regularity of the minimizer, but not on the locations of singularities or other properties. 
Extensive numerical experiments presented in \cite{feng2016enhanced, schnake2017numerical} are quite promising, as they all point to the convergence of this enhanced finite element method. However, to settle the convergence issue, a rigorous proof of the convergence of the minimizer of \eqref{enhanced} to the minimizer of \eqref{Eq: Mania} is necessary. One pertinent mathematical framework is the $\Gamma$-convergence theory \cite{dal2012introduction}. 

This framework has appeared in many instances in the calculus of variations, including studying the convergence of discretizations for finding equilibrium states of nematic liquid crystals through the Landau-de Gennes Model; see \cite{borthagaray2020structure, nochetto2017finite, nochetto2022gamma}. This method is an alternative error estimation technique for problems where compactness techniques, through a priori estimates, may not be readily available.
  
Meanwhile, the $\Gamma$-convergence analysis of \eqref{enhanced} was indeed pursued and partially succeeded in \cite{feng2016enhanced, schnake2017numerical}. Specifically, these works attempted to prove $\Ga$-convergence of $J^{\al}_h$ to $J$ with respect to the weak $W^{1, 1}$ topology as $h \rightarrow 0^+$, for some fixed $\al > 0$, and the {\em lim-inf} inequality was successfully established  (see \cite[Lemma 4.1]{schnake2017numerical}). 
However, the construction of a suitable recovery sequence has remained open. It should be noted that a $\Ga$-limit necessarily exists for at least a sub-sequence of $\{J^{\al}_h\}_{h > 0}$ owing to the compactness of $\Gamma$-convergence \cite[Proposition 1.42]{braides2002gamma}.

 The first main objective of this paper is to pick up where \cite{feng2016enhanced, schnake2017numerical} stalled and to complete the $\Ga$-convergence proof for the approximate functional $J^{\al}_h$, albeit with respect to the strong  
$W^{1, p}$-topology. We may use the standard finite element nodal interpolant as the recovery sequence, but we also leverage the often overlooked higher-order differentiability of continuous piecewise polynomials. Indeed, $X_h \subset W^{1 + s, p}$ for $0\leq s p< 1$ and $1< p\leq \infty$, which has been known for $p = 2, \infty$  (see \cite{bramble1991analysis, xu1989theory}). 
 We extend this result to all $p>1$. 
 Moreover, since  the minimizer $\overline{u}(x) = x^{\frac{1}{3}}$ of \eqref{Eq: Mania} belongs to the fractional Sobolev space $W^{1 + s, p}$ for $0\leq s<1/3$ and $1\leq p< 3/2$ when $(2/3 + s)p < 1$,
 It is natural to consider the fractional Sobolev spaces. 
 However, although $X_h \subset W^{1 + s, p}$,  $X_h$ is not dense in $W^{1 + s, p}$ unless $s=0$ 
 which seems to pose a technical difficulty in establishing the $\Ga$-convergence in the strong   $W^{1, p}$-topology.
 %
It turns out that  this difficulty can be mitigated because the density is not necessary. Moreover, the additional
regularity of the minimizer justifies regarding problem \eqref{Eq: Mania} as  the minimization of $J$ over $X:=\Alg\cap W^{1 + s, p}(0, 1)$ rather than over $\Alg$, which,  combined with \eqref{cutoffMania} and the explicit $O(h^{-s})$ blowup rate in the $W^{1 + s, p}$-norm for the finite element function (thanks to the generalized inverse inequality, see \eqref{Eq: InverseFracGen}), mitigates singularities in the derivative. This is sufficient to complete the $\Ga$-convergence proof.  It should be emphasized that, although this recasting of the problem is crucial for our convergence analysis, it has no impact on the numerical method
  and its implementation. Therefore, it is only an analytical technique for the convergence proof.
 
 	The second main objective is to extend the enhanced finite element method and the $\Gamma$-convergence result to the 
 	two-dimensional Foss's problem from nonlinear elasticity, which is also known to possess the LGP. Recall that Foss's functional is
\begin{equation}\label{Eq: Foss}
    F(v) \ := \ 66\left(\frac{13}{14}\right)^{14}\int_{\Om}\left(\frac{y}{y - 1}\right)^{14}|v|^{\frac{14 - 3y}{y - 1}}(|v|^{\frac{y}{y - 1}} - x)^2(v_x)^{14}dxdy,
\end{equation}
originally introduced in \cite{foss2001lavrentiev, foss2003lavrentiev}. Here $\Om := (0, 1) \times \left(\frac{3}{2}, \frac{5}{2}\right) \subset \R^2$. The admissible space for this functional is
\begin{equation*}
    \Bor \ := \ \{v \in W^{1, 1}(\Om) \cap L^{\infty}(\Om) \ | \ v(0, y) = 0, \ v(1, y) = 1, \ F(v) < \infty\},
\end{equation*}
and we will discuss other properties of this problem in Section \ref{facts}. For now, we simply remark that both \eqref{Eq: Mania} and \eqref{Eq: Foss} contain factors involving the input function multiplied by powers of its first-order derivatives, allowing us to use the same blueprint to navigate the technical details of the required convergence proofs.

The remainder of this paper is organized as follows. In Sectiosn \ref{Sec: notationAndPreliminaries} and \ref{eFEM}, we introduce the required space and function notation, the definition of finite element spaces, and a few facts about Mani\`a's and Foss's problems.  Section \ref{Sec: GaConv} is devoted to proving the desired  $\Ga$-convergence result for Mani\'a's problem, while Section \ref{Sec: GaConvFoss} carries out the same task for Foss's problem. Next, Section \ref{Sec: Numerics} presents additional numerical result beyond that provided in \cite{feng2016enhanced, schnake2017numerical}. Finally, we discuss other LGP problems to be explored in Section \ref{Sec: future}. 



\section{Preliminaries}\label{Sec: notationAndPreliminaries}


\subsection{Function spaces}\label{Subsec: FuncSpaces}

Let $\Omega\subset \mathbb{R}^n$ be a bounded domain. For integer $k\geq 0$ and real number $1\leq p <\infty$, $W^{k,p}(\Omega)$ denotes the Sobolev space consisting of all functions whose weak partial derivatives up to order $k$ belong to $L^p(\Omega)$. Note that $W^{0,p}(\Omega)=L^p(\Omega)$. We also define, for any real number $r > 0$, setting $k:=[r]$,   the Fractional Sobolev norm by 
\begin{equation}\label{Eq: Hr}
    \|v\|^p_{W^{r, p}(\Om)}  :=  \|v\|^p_{W^{k, p}(\Om)} + \sum_{|\al| = k}\int_{\Om}\int_{\Om}\frac{|\pa^{\al}v(x) - \pa^{\al}v(y)|^p}{|x - y|^{n + p(r - k)}}dxdy.
\end{equation}
 The second term on the right-hand side will be abbreviated as $[v]^p_{W^{k, p}(\Om)}$, and when $r=k$ is an integer, the definition \eqref{Eq: Hr} becomes 
\begin{equation}\label{Eq: Hk}
    \|v\|^p_{W^{k, p}(\Om)} :=  \sum_{|\al| \leq k}\|\pa^{\al}v\|^p_{L^p(\Om)}.
\end{equation}
Then $W^{r, p}(\Om)$ denotes the Banach space endowed with the norm defined in \eqref{Eq: Hr}.  When $p=2$, we set $H^r(\Omega):= W^{r,2}(\Omega)$. 
 For $0 < s < 1$ and $1 \leq p \leq \infty$, we have 
	\begin{equation}\label{W^1,p}
		W^{1 + s, p}(\Om)  :=  \{u: \Om \rightarrow \R \ | \ u \in L^p(\Om), \ \grad u \in W^{s, p}(\Om; \R^n)\}.
	\end{equation}
These spaces are further discussed in \cite{adams2003sobolev, demengel2012functional, di2012hitchhikers, leoni2023first}.
In this paper,  we only consider two special cases of domains: (i) $n=1$ and $\Omega=(0,1)$; (ii) $n=2$ and   $\Om = (0, 1) \times \left(\frac{3}{2}, \frac{5}{2}\right)$. 

Throughout this paper, $C$ will be used to denote a generic positive constant that is independent of the mesh size $h$. The notation $\alpha\lesssim \beta$ stands for $\alpha\leq C \beta$ for some $C>0$. For any $1<p<\infty$, we define its conjugate  as $p':=p/(p-1)$. 
 
\subsection{Facts about Mani\`a's and Foss's problems}\label{facts}
Recall that the Mani\`a's and Foss's functionals $J$ and  $F$ are given by  \eqref{Eq: Mania} and \eqref{Eq: Foss}, respectively.  Their respective admissible spaces are $\Alg$ and $\Bor$, namely,
\begin{eqnarray*}
\Alg &:=& \{u \in W^{1, 1}(0, 1) \ | \ u(0) = 0, \ u(1) = 1, \ J(u)<\infty\},   \\
\Bor & := & \{v \in W^{1, 1}(\Om) \cap L^{\infty}(\Om) \ | \ v(0, y) = 0, \ v(1, y) = 1, \ F(v) < \infty\}, 
\end{eqnarray*} 
where $\Om := (0, 1) \times \left(\frac{3}{2}, \frac{5}{2}\right)$ is in the definition of $\Bor$. 
	
It has already been noted in Section \ref{Sec: intro} that  the minimizer for Mani\`a's problem is  $\overline{u}(x) = x^{\frac{1}{3}}$, and it is easy to verify that $\overline{u}$ belongs to the fractional Sobolev space $W^{1 + s, p}$ for $0\leq s<1/3$ and $1\leq p< 3/2$ where $(2/3 + s)p < 1$.  
For Foss's problem, it is known that the  minimizer is $\overline{v}(x, y):= x^{\frac{y - 1}{y}}$ and 
$\overline{v}$ belongs to $W^{1 + s, p}(\Om)$ for $1 \leq p < \frac{3}{2}$ and $s < \frac{1}{p} - \frac{2}{3}$. 
We will prove the latter claim in Appendix \ref{appendixA}. It should be  noted that the exact forms or expressions of these minimizers are not required or used in our enhanced finite element methods and convergence analysis. The only information about the minimizers used in our convergence analysis (to be given in the subsequent sections) is a slightly higher differentiablity than $W^{1,1}$ offers.

\section{Enhanced finite element methods and fractional inverse inequalities} \label{eFEM} 

\subsection{Finite element spaces and interpolations}\label{FE-spaces}
Let $N>>1$ be an integer, define $h=1/N$ and $x_j:=jh$ for $j=0,1,2,\cdots, N$. Then, $\{ x_j\}_{j=0}^N$ is a uniform mesh over $[0,1]$ with mesh size $h$.  Also introduced in Section \ref{Sec: intro}, $X_h$ denotes the conforming linear finite element space over the mesh  $\{ x_j\}_{j=0}^N$ that preserves the boundary conditions of $\Alg$.  Specifically, 
\begin{equation}\label{CtsPWLin}
    X_h := \ \bigl\{v_h \in C(0, 1)\ | \  \ v_h|_{I_k} \in \mathcal{P}_1(I_k) \,\, \fa 0\leq k \leq N - 1\,, \ v_h(0) = 0, \ v_h(1) = 1 \bigr\},
\end{equation}
where $I_k:=(x_k, x_{k+1})$ and $\mathcal{P}_1(I_k)$ denote the set of all linear functions on $I_k$. 

Similarly, let $\{\Triag_h\}_{h > 0}$ be a family of quasi-uniform triangular meshes over the rectangular  domain 
$\Om = (0, 1) \times \left(\frac{3}{2}, \frac{5}{2}\right)$ with mesh size $h>0$, and let $Y_h$ denote the conforming linear finite element space over $\Triag_h$ that preserves the boundary conditions of $\Bor$, which is the two-dimensional analog of \eqref{CtsPWLin}, namely
\begin{equation}\label{Eq: CtsPWLin2D}
	Y_h \ := \ \{v_h \in C(\Om) \ | \ v_h|_T \in \mathcal{P}_1(T) \ \fa T \in \Triag_h, \ v_h(0, \cdot) \ = \ 0, \ v_h(1, \cdot) \ = \ 1\},
\end{equation}
where $\mathcal{P}_1(T)$ denotes the set of linear functions on $T$. Notice $Y_h \subset \Bor$ for every $h > 0$.

Let $\{\phi_m\}^{M}_{m= 1}$ denote the nodal basis of the finite element space $\Alg$ (resp. $\Bor$) with $M=\mbox{dim}(X_h)$ (resp. $M=\mbox{dim}(Y_h)$). For each $v \in C^0(\Omega)$,  we define its nodal interpolation  $I_hv\in X_h$ (or $Y_h$) as follows: 
\begin{equation}\label{Eq: Noda-xl}
	I_h v(x) \ := \ \sum^{M}_{m=1}\phi_m(x)v(x).
\end{equation}
Then the following properties of $I_h$ are well-known (cf.  \cite{ern2004theory}).

\begin{lemma}\label{lemma3.1} 
 	Let $\Omega=(0,1)$ or $\Omega=(0, 1) \times \left(\frac{3}{2}, \frac{5}{2}\right)$ and $v \in W^{r, p}(\Om)$ for $0 < r \leq 2$ and $1 \leq p < \infty$,  then $I_h$ satisfies  the following properties:
 \begin{eqnarray} \label{property-1}
 		\|I_h v\|_{L^{\infty}(\Om)}  &\lesssim& \|v\|_{L^{\infty}(\Om)}, \qquad\mbox{if  $1<r\leq 2$},  \\
 	 \label{property-2}
 		\|v - I_hv\|_{L^p(\Om)} &\lesssim& h^r[v]_{W^{r, p}(\Om)},\\
   \label{property-3}
 		\|v - I_hv\|_{W^{1, p}(\Om)} &\lesssim& h^{r-1} [v]_{W^{r, p}(\Om)} \qquad\mbox{if  $1<r\leq 2$}. 
 \end{eqnarray}
 If we write $r=1+s$ for $0< s\leq 1$, the right-hand side of \eqref{property-3} becomes $h^s[v]_{W^{1 + s, p}(\Om)}$. 
\end{lemma}

We note that all functions in $\Alg \cap W^{1 + s, p}(\Omega)$ and $\Bor\cap W^{1 + s, p}(\Omega)$ are H\"older continuous; hence, they belong to $L^{\infty}(\Omega)$ (cf. \cite{leoni2024first}).  Therefore, \eqref{property-1} makes sense. 

\smallskip
 \begin{remark} \label{rem3.1}
 For our $\Gamma$-convergence analysis to be presented later, we will not need the special structure of the finite element interpolant, but rather just its stability and approximation properties. Therefore, $I_h$ can be replaced by any operator $K_h: \Alg \ (\mbox{\rm resp. } \Bor) \to X_h \ (\mbox{\rm resp. } Y_h)$,  which satisfies \eqref{property-1}--\eqref{property-3}. 
%
\end{remark} 

\subsection{Enhanced finite element methods} \label{eFEMs} 
We first recall that the enhanced finite element method for Mani\`a's problem is defined in \eqref{enhanced}. It minimizes the numerical energy functional $J^\alpha_h$ (see  \eqref{cutoffMania})  over the finite element space $X_h$.  

To define the enhanced finite element method for Foss's problem,  let $\be > 0$ be a fixed constant.  we define a discrete analog $F^{\be}_h: Y_h \rightarrow \R$ of $F$ as follows: 
\begin{equation}\label{Eq: FossCutoff}
    F^{\be}_h(v) := 66\left(\frac{13}{14}\right)^{14}\int_{\Om}\left(\frac{y}{y - 1}\right)^{14}|v|^{\frac{14 - 3y}{y - 1}}(|v|^{\frac{y}{y - 1}} - x)^2(\chi^{\al}_h v_x)^{14}\, dxdy.
\end{equation}
Then, the enhanced finite element method for Foss's problem is defined as 
\begin{equation}\label{Eq: enhancedFoss}
	\overline{v}_h = \underset{v_h\in Y_h}{\argmin} \, F^{\be}_h(v_h). 
\end{equation}

\begin{remark}
We note that  since $Y_h\subset W^{1,\infty}(\Omega)$, the Lavrentiev gap implies that  (cf. \cite[Theorem 6.1]{foss2003examples})
\begin{equation}\label{Eq: FossLavGap}
	\inf_{v \in \Bor}F(v) <	\inf_{v \in W^{1,\infty}} F(v)  \leq \inf_{v_h \in Y_h}F(v_h). 
\end{equation}
Thus, the standard finite element method does not converge for Foss's problem, which is expected.  However, unlike in Mani\`a's problem, we do not automatically gain that admissible functions for $F(\cdot)$ belong to $L^{\infty}(\Om)$ through Sobolev embeddings (especially, $\overline{v} \notin W^{1, p}(\Om)$ for any $p > 2$), so we make this inclusion explicit in the definition of $\Bor$. In turn, we may include the case $p = 1$ here.
\end{remark}


\subsection{Fractional inverse inequalities}\label{Subsec: FracInverse}
The goal of this section is to establish general inverse inequalities for finite element functions with respect to the norm of the fractional Sobolev space $W^{s,p}$. These results will be crucially used in our $\Gamma$-convergence proofs. Moreover, they also have a great deal of independent interest in numerical PDEs. To this end, we first state our mesh assumptions required to ensure our fractional inverse inequalities

\begin{assumption}[Mesh assumptions for inverse inequalities]\label{Assump: Inverse}
    For the setup of our general inverse inequalities, we make the following mesh assumptions:
    \begin{enumerate}\label{Enum: In}
        \item[\rm (a)] Let $\Om \subset \R^n$ be a bounded polygonal domain with boundary $\pa \Om$.
         \item[\rm (b)] Let $\{\Triag_h\}_{h > 0}$ be a shape-regular family of meshes with shape regularity parameter $\sig$.  Since $\Om$ is polygonal, then $\Om = \cup_{T \in \Triag_h}T$.
        \item[\rm (c)]  Define the mesh size $h := \max_{T \in \Triag_h}\diam(T)$ for the mesh $\Triag_h$ (where $\diam(T)$ denotes the diameter of triangle $T$).
    \end{enumerate}
\end{assumption}
We note that the assumption of $\Om$ being polygonal is not strictly necessary and is included for the sake of simplicity. Notice also that the setups provided in Subsections \ref{FE-spaces} and \ref{eFEMs} satisfy these assumptions. Also, a fractional inverse inequality is already known in the case $p = 2$; its proof can be found in  \cite[Theorem 3.1]{xu1989theory}. For completeness, we quote the result in the following theorem. 

\begin{theorem}[Fractional inverse inequality, $p = 2$]\label{Thm: InverseInequality}
    Let $m \in \mathbb{N}^+$ and $v \in \mathcal{P}_m(\Om)$, the set of continuous piecewise polynomials of degree at most $m$ with respect to the mesh $\Triag_h$. Then, for any $0 \leq \gamma < 1/2$,  there holds the inequality
    \begin{equation}\label{Eq: XuInverse}
        [v]_{H^{1 + \gamma}(\Om)} \ \lesssim \ h^{-\gamma}\|v\|_{H^1(\Om)}.
    \end{equation}
\end{theorem}

We note that the above estimate immediately implies that $Y_h\subset H^{1+\gamma}(\Om)$. However, it is well-known that $Y_h$ is not dense in $H^{1+\gamma}(\Om)$ for $\gamma >0$.
This estimate is most relevant in the case where the function to be approximated belongs only to $H^1$. Since we know a priori that the minimizers of $J$ and $F$ are better than the $H^1$-functions (see Subsection \ref{facts}),  we need to deal with functions in certain fractional Sobolev spaces $W^{s, p}$. Thus, we want to generalize Theorem \ref{Thm: InverseInequality} to the fractional Sobolev spaces $W^{s, p}$, and present the results on a more general domain and family of meshes for the sake of enhancing versatility.

\smallskip

\begin{theorem}[Generalized fractional inverse inequality]\label{Thm: GenFracInverse}
    Let $m \in \mathbb{N}^+$ and $v \in \mathcal{P}_m(\Om)$, the set of continuous piecewise polynomials of degree at most $m$ with respect to the mesh $\Triag_h$. Suppose that Assumption \ref{Assump: Inverse} holds and that $0<s<1$, $1 \leq p < \infty$ are such that $sp < 1$. Then there holds
    \begin{equation}\label{Eq: GenFracInverse}
        [v]_{W^{s, p}(\Om)} \ \lesssim \ h^{-s}\|v\|_{L^p(\Om)}.
    \end{equation}
\end{theorem}

\begin{proof}
    We follow and adapt the idea of the proof of Theorem \ref{Thm: InverseInequality}. We start with decomposing the semi-norm (defined in \eqref{Eq: Hr}) as the following sum:
    \begin{equation}\label{Eq: GenFracInverse1}
        [v]^p_{W^{s, p}(\Om)} = I + II + III,
    \end{equation}
    where 
     \begin{eqnarray}\label{Eq: GenFracInverse2}
     	\begin{aligned}
     		I &:=  \mathop{\sum\sum}_{\substack{T, T' \in \Triag_h \\ T \cap T' = \emptyset}}\int_T\int_{T'}\frac{|v(x) - v(y)|^p}{|x - y|^{n + sp}}\,dxdy , \\
     		II &:= \mathop{\sum\sum}_{\substack{T, T' \in \Triag_h, \ T \neq T' \\ T \cap T' \neq \emptyset}}\int_T\int_{T'}\frac{|v(x) - v(y)|^p}{|x - y|^{n + sp}}\,dxdy,\\
     		III &:= \sum_{T \in \Triag_h}\int_T\int_{T}\frac{|v(x) - v(y)|^p}{|x - y|^{n + sp}}\,dxdy.
     	\end{aligned}
     \end{eqnarray}
    For $I$, we use the symmetry of the integrand and the elementary inequality $(a + b)^p \lesssim a^p + b^p$ for $a, b > 0$ to obtain
    \begin{eqnarray}\label{Eq: GenFracInverse3}
    \begin{aligned}
        I &\lesssim \mathop{\sum\sum}_{\substack{T, T' \in \Triag_h \\ T \cap T' = \emptyset}}\int_T\int_{T'}\frac{|v(x)|^p}{|x - y|^{n + sp}}\, dxdy \\ 
        &\lesssim \sum_{T \in \Triag_h}\int_T |v(x)|^p \left(\int_{|y - x| \geq \sig^{-1}h}\frac{1}{|x - y|^{n + sp}}dy\right)dx \\
        &\lesssim \ h^{-sp}\|v\|^p_{L^p(\Om)} .
        \end{aligned}
    \end{eqnarray}
    Now, to bound $II$, we let $T \in \Triag_h$ and fix $x \in T$, then define $d_x := \dist\{x, \pa T\}$. Using polar coordinates,
    \begin{eqnarray}\label{Eq: GenFracInverse4}
        \sum_{\substack{T, T' \in \Triag_h, \ T \neq T' \\ T \cap T' \neq \emptyset}}\int_{T'}\frac{dy}{|x - y|^{n + sp}}  &\leq  \int_{\R^n\setminus B(0, d_x)}\frac{dy}{|x - y|^{n + sp}} \\
        &\lesssim \int^{\infty}_{d_x}\frac{r^{n - 1}}{r^{n + sp}}\, dr \lesssim d_x^{-sp}. \nonumber
    \end{eqnarray}
    In addition, by the definition of $d_x$ and the fact that $sp < 1$, we have
    \begin{equation}\label{Eq: GenFracInverse5}
        \int_T d_x^{-sp}dx  \lesssim \ h^{n - sp}.
    \end{equation}
    Combining \eqref{Eq: GenFracInverse4} and \eqref{Eq: GenFracInverse5} alongside standard inverse inequalities gives
    \begin{equation}\label{Eq: GenFracInverse6}
        II \ \lesssim h^{n - sp}\sum_{T \in \Triag_h}\|v\|^p_{L^{\infty}(T)} \lesssim \ h^{-sp}\|v\|^p_{L^p(\Om)}.
    \end{equation}
    
    Finally,  $III$ can also be handled using standard inverse inequalities as follows:
    \begin{eqnarray}\label{Eq: GenFracInverse7}
        III &\leq& \sum_{T \in \Triag_h}[v]^p_{W^{1, \infty}(T)}\int_T\int_T\frac{|x - y|^p}{|x - y|^{n + sp}}\, dxdy \\
        &\lesssim& h^{n + p - sp}\sum_{T \in \Triag_h}[v]^p_{W^{1, \infty}(T)} \lesssim \ h^{-sp}\|v\|^p_{L^p(\Om)}. \nonumber
    \end{eqnarray}
    Combining \eqref{Eq: GenFracInverse3}, \eqref{Eq: GenFracInverse6}, and \eqref{Eq: GenFracInverse7}, and taking $p$th roots gives the desired result.
\end{proof}

The following corollary follows from applying Theorem \ref{Thm: GenFracInverse} in the 1-D case, to the weak derivative of functions in $W^{1 + s, p}(0, 1)$ for some $0 
< s < 1$ and $1 \leq p < \infty$.

\begin{corollary}\label{Cor: FracHigherOrder}
    Suppose that $0<s < 1$ and $1 \leq p < \infty$ satisfy $sp < 1$. Let $v_h \in X_h$  as defined in \eqref{CtsPWLin}. Then there holds the following inverse inequality: 
    \begin{equation}\label{Eq: InverseFracGen}
        [v_h]_{W^{1 + s, p}(0, 1)} \ \lesssim \ h^{-s}\|v_h\|_{W^{1, p}(0, 1)}.
    \end{equation}
    In particular, the inequality implies that $v_h \in W^{1 + s, p}(0, 1)$.
\end{corollary}

We conclude this section by noting that when $p=2$, the condition $sp<1$ implies that $s< 1/2$. Hence, \eqref{Eq: InverseFracGen} recovers \eqref{Eq: XuInverse}.  Moreover, 	\eqref{Eq: InverseFracGen} also provides a blowup rate $O(h^{-s})$  for the finite element functions in the $W^{1+s,p}$-norm. This estimate will play a critical role in our convergence analysis, see Lemma \ref{Lem: ErrorEnergyIV}.


\section{$\Ga$-convergence for Mani\`a's problem}\label{Sec: GaConv}

One of the main goals of this paper is to complete the proof of $\Ga$-convergence of the approximate functional 
$J_h^\alpha$,  defined in \eqref{cutoffMania}, to the Mani\`a's functional $J$ defined in \eqref{Eq: Mania}. Consequently, we will be able to conclude the convergence of the numerical minimizers. We recall that in \cite{feng2016enhanced, schnake2017numerical}, the authors were able to prove a {\em lim-inf} inequality for the approximate functional, but did not construct a recovery sequence converging in a suitable topology; this is the gap we intend to fill in here. We stress that our proof technique heavily relies on the structure of \eqref{Eq: Mania}, but not on the exact knowledge of the true minimizer itself. Throughout this section, let $K_h$ be an operator which satisfies properties \eqref{property-1}--\eqref{property-3} ( see Remark \ref{rem3.1}).

First, we recall the definition of $\Ga$-convergence tailored to Mani\`a's problem. 

\smallskip
\begin{definition}[$\Ga$-convergence]\label{GaConvDef}
Let $\alpha > 0$,  $0<s<1/3$ and $1 < p <3/2$ be such that $\left(2/3 + s\right)p < 1$, $X_h$ denote the piecewise linear finite element space, and  $X:=\mathcal{A} \cap W^{1 + s,p}(0,1)$. 
We say that the family of functionals $\{J^{\al}_h\}_{h>0}$, defined on $X_h$,  {\em $\Ga$-converges} as $h\to 0^+$ to the functional $J$, defined on $X$,  (written as $J^{\al}_h \xrightarrow{\Ga} J$) with respect to the strong  $W^{1, p}(0, 1)$-topology if the following hold:
    \begin{enumerate} \label{GC1}   
        \item[\rm (i)] \underline{The lim-inf property:} If $\{v_h\}_{h > 0} \subset X_h$ 
        is such that $v_h \rightarrow v$ strongly in $W^{1, p}(0, 1)$ then we have the {\em lim-inf inequality}
        \begin{equation}\label{liminfIneqCutoff}
           \underset{h\to 0^+}{ \mbox{\rm lim\,inf}}\,  J^{\al}_h(v_h)  \geq J(v).
        \end{equation}\label{GC2}
        \item[\rm (ii)] \underline{Recovery sequence property:} For any $v \in X$, 
        there exists a {\em recovery sequence} $\{v_h\}_{h > 0}\subset X_h$ such that $v_h \rightarrow v$ strongly in $W^{1, p}(0, 1)$ and the following {\em lim-sup inequality} holds:
        \begin{equation}\label{limsupneqCutoff}
         \underset{h\to 0^+}{ \mbox{\rm lim\,sup}}\, J^{\al}_h(v_h)   \leq J(v).
        \end{equation}
    \end{enumerate}
\end{definition}

Recall that a property stronger than (i) was proven in \cite{feng2016enhanced, schnake2017numerical}, namely assuming convergence only in the weak $W^{1, 1}(0, 1)$-topology. Also notice that $X_h \subset W^{1 + s, p}(0, 1)$ provided that 
$0\leq sp<1$ (see Corollary \ref{Cor: FracHigherOrder}). Our goal here is to establish property (ii). The next technical lemma is another step towards achieving our goal. 


\begin{lemma}\label{Lem: ErrorEnergyIII}
    Let $1 < p < \infty$, $0 < s < 1$ and suppose $v \in X:= \Alg \cap W^{1 + s, p}(0, 1)$. Let $0 < \al < (1 + s)/6$ be fixed, and  $v_h:=K_hv$ 
   Then there holds
    \begin{eqnarray}\label{Eq: ErrorEnergyIII}
        &\int^{1}_{0} \bigl(\chi^{\al}_h(v_h'(x)\bigr)^6 \bigl[(v_h(x)^3 - x)^2 - (v(x)^3 - x)^2 \bigr]\,dx \\ &\hskip 1in  \lesssim h^{1 + s - 6\al}\|v\|^5_{L^{\infty}(0, 1)}[v]_{W^{1 + s, p}(0, 1)}. \nonumber
    \end{eqnarray}
\end{lemma}

\begin{proof}
    By the definition of $\chi_h^\alpha$, we have
    \begin{eqnarray*} 
     && \int^{1}_{0} \bigl(\chi^{\al}_h(v_h'(x)\bigr)^6 \bigl[(v_h(x)^3 - x)^2 - (v(x)^3 - x)^2 \bigr]\,dx   \\
      &&\hskip 0.9in  \leq  h^{-6\al}\int^{1}_{0} \bigl|(v_h(x)^3 - x)^2 - (v(x)^3 - x)^2\bigr|\, dx \nonumber \\
       &&\hskip 0.9in  \lesssim h^{-6\al}\int^{1}_{0} \bigl|v(x)^3 + v_h(x)^3 - 2x\bigr|\cdot \bigl|K_hv(x)^3 -  v(x)^3 \bigr|\, dx. \nonumber
    \end{eqnarray*}
    
    Since $v \in W^{1, p}(0, 1)$ for $p > 1$, then $v \in L^{\infty}(0, 1)$, and in turn $v_h \in L^{\infty}(0, 1)$ due to \eqref{property-1}, so
    \begin{eqnarray*} 
       && \int^{1}_{0} \bigl(\chi^{\al}_h(v_h'(x)\bigr)^6 \bigl[(v(x)^3 - x)^2 - (v_h(x)^3 - x)^2 \bigr]\,dx \\
       &&\hskip 1.2in \lesssim h^{-6\al}\|v\|^3_{L^{\infty}(0, 1)} \int^{1}_{0} \bigl|v(x)^3 - v_h(x)^3 \bigr|\, dx. \nonumber
    \end{eqnarray*}
    By using a binomial formula and (again) that $v, v_h \in L^{\infty}(0, 1)$, we obtain
    \begin{eqnarray}\label{Eq: ErrorEnergyIIIEq4}
       && \int^{1}_{0} \bigl(\chi^{\al}_h(v_h'(x)\bigr)^6 \bigl[(v(x)^3 - x)^2 - (v_h(x)^3 - x)^2 \bigr]\,dx  \\
      && \hskip 1.8in \lesssim  h^{-6\al}\|v\|^5_{L^{\infty}(0, 1)} \|v-v_h\|_{L^1(0, 1)}. \nonumber
    \end{eqnarray}
   Then, \eqref{Eq: ErrorEnergyIII} follows from \eqref{property-2} and \eqref{Eq: ErrorEnergyIIIEq4}; the proof is complete.
\end{proof}

We now state and prove a second technical lemma.

\begin{lemma}\label{Lem: ErrorEnergyIV}
Let $1 < p < \infty$, $0 < s < 1$, and suppose that $v \in X:=\Alg \cap W^{1 + s, p}(0, 1)$. Let $0 < \al < s/5$ be fixed, and  $v_h:=K_hv$.
  Then there holds
    \begin{eqnarray}\label{Eq: ErrorEnergyIII2}
       & \int^1_0|(\chi^{\al}_h(v_h'(x))^6 - (\chi^{\al}_h(v'(x)))^6|(v(x)^3 - x)^2 dx \\
        &\hskip 1.5in  \lesssim \ h^{s - 5\al}  
      \|v\|^6_{L^{\infty}(0, 1)}[v]_{W^{1 + s, p}(0, 1)} . \nonumber
    \end{eqnarray}    
\end{lemma}

\begin{proof}
    By factoring the difference of sixth powers, the definition \eqref{cutoff}, the $L^{\infty}$-stability of $K_h$ granted by \eqref{property-1}, and the continuous embedding $W^{1 + s, p}(0, 1) \subset L^{\infty}(0, 1)$, we have
    \begin{eqnarray}\label{Eq: ErrorEnergyIII1}
    \begin{aligned}
        & \int^1_0|(\chi^{\al}_h(v_h'(x))^6 - (\chi^{\al}_h(v'(x)))^6|(v(x)^3 - x)^2 dx \\ 
        &\lesssim \int^1_0|\chi^{\al}_h(v_h'(x)) - \chi^{\al}_h(v'(x))| \cdot (|\chi^{\al}_h(v_h'(x))|^5 + |\chi^{\al}_h(v'(x))|^5)(v(x)^3 - x)^2dx \nonumber\\ 
        &\leq \ h^{-5\al}\|v\|^6_{L^{\infty}(0, 1)}\int^1_0|\chi^{\al}_h(v_h'(x)) - \chi^{\al}_h(v'(x))|dx.
         \end{aligned}
    \end{eqnarray}
    Now, due to the inequality $|\chi^{\al}_h(v_h'(x)) - \chi^{\al}_h(v'(x))| \leq |v_h'(x) - v'(x)|$ and \eqref{property-3}, we obtain the desired result.
\end{proof}

Notice this lemma would not stand if we only assumed that $v \in \Alg \cap W^{1, p}(0, 1)$.

Finally, we are ready to state and prove the main result of this paper for Mani\`a's problem. Since the minimizer for the continuous problem is known to belong to $W^{1 + s, p}(0, 1)$ 
for $0<s<1/3$ and $1 < p <3/2$ such that $\left(2/3 + s\right)p < 1$, the following result is stated specifically in that case, 

\begin{theorem}[$\Ga$-convergence of $J_h^\alpha$]\label{Th: GaConvCutoffs}
    Let $0<s<1/3$ and $1 < p <3/2$ be such that $\left(2/3 + s\right)p < 1$. Assume that $\al < \min\{(1 + s)/6, s/5\}$ is fixed, and we have that $J^{\al}_h \xrightarrow{\Ga} J$ as $h\to 0^+$ in the strong $W^{1, p}(0, 1)$-topology. 
\end{theorem}

\begin{proof}
   Since strong convergence in $W^{1, p}(0, 1)$ implies weak convergence in $W^{1, 1}(0, 1)$, then \cite[Theorem 4.1]{schnake2017numerical} implies that the inequality \eqref{liminfIneqCutoff} holds. In our specific setting, we may expedite the proof. Let $\{v_h\}_{h > 0}$ be a sequence such that $v_h \in X_h$ for all $h > 0$ and $v_h \rightarrow v$ strongly in $W^{1, p}(0, 1)$. Then $\chi^{\al}_h(v_h'(x))^6(v_h(x)^3 - x)^2 \rightarrow v'(x)^6(v(x)^3 - x)^2$ for a.e. $x \in (0, 1)$. This, combined with the non-negativity of the integrands for $\{J^{\al}_h\}_{h > 0}$, allows us to conclude \eqref{liminfIneqCutoff} by Fatou's Lemma.

    Next, to prove inequality \eqref{limsupneqCutoff},  we fix $v \in X$ and let $\{v_h\}_{h > 0} = \{K_h v\}_{h > 0}$.  We now show that this choice is a valid recovery sequence. It immediately follows from \eqref{property-3} that $v_h \rightarrow v$ strongly in $W^{1, p}(0, 1)$. For any $h > 0$ fixed, we have 
\begin{eqnarray}\label{Th: GaConvCutoffsEq1}
        J^{\al}_h(v_h) &= & \int^{1}_{0}(\chi^{\al}_h \bigl(v_h'(x) \bigr)^6   \bigl[(v_h(x)^3 - x)^2 - (v(x)^3 - x)^2\bigr]\, dx  \\ 
        && \quad +\int^{1}_{0}(\bigl(\chi^{\al}_h(v_h'(x) \bigr)^6 - \bigl(\chi^{\al}_h(v'(x) \bigr)^6) (v(x)^3 - x)^2\, dx  \nonumber\\
        &&\quad  + \int^{1}_{0}\bigl(\chi^{\al}_h(v'(x) \bigr)^6 (v(x)^3 - x)^2\, dx.
         \nonumber
    \end{eqnarray}
    The first term on the right-hand side of \eqref{Th: GaConvCutoffsEq1} converges to $0$ thanks to Lemma \ref{Lem: ErrorEnergyIII}. 
    As for the second term, we observe that 
    \begin{eqnarray}\label{Th: GaConvCutoffsEq2}
      && \int^{1}_{0}(\bigl(\chi^{\al}_h(v_h'(x) \bigr)^6 - \bigl(\chi^{\al}_h(v'(x) \bigr)^6) (v(x)^3 - x)^2\, dx  \\  
       &&\hskip 1.0in \leq \int^1_0|(\chi^{\al}_h(v_h'(x))^6 - (\chi^{\al}_h(v'(x)))^6|(v(x)^3 - x)^2 dx,  \nonumber
    \end{eqnarray}
    and the right-hand side converges to $0$ by Lemma \ref{Lem: ErrorEnergyIV}. Finally, it is evident that
    \begin{equation}\label{Th: GaConvCutoffsEq3}
        \int^{1}_{0}\bigl(\chi^{\al}_h(v'(x) \bigr)^6 (v(x)^3 - x)^2dx \ \leq \ J(v)
    \end{equation}
     for any $h > 0$. Combine \eqref{Th: GaConvCutoffsEq1}, \eqref{Th: GaConvCutoffsEq2}, and \eqref{Th: GaConvCutoffsEq3}, and send $h \rightarrow 0^+$ to conclude.
\end{proof}

We remark that this result is slightly stronger than the one anticipated in \cite{feng2016enhanced} because the convergence of the recovery sequence takes place not only in the weak $W^{1, 1}(0, 1)$ topology but also in the strong $W^{1, p}(0, 1)$ topology. 

We conclude this section by proving the convergence of 
the functionals $\{J^{\al}_h(\cdot)\}_{h > 0}$ at the corresponding minimum values. If the functionals $\{J^{\al}_h(\cdot)\}_{h > 0}$ were equi-coercive in $W^{1 + s, p}(0, 1)$, then we could immediately conclude that $\overline{u}_h\rightarrow \overline{u}$ strongly in $W^{1, p}(0, 1)$ (see \cite[Corollary 7.20]{dal2012introduction} or \cite[Theorem 13.3]{Rin}). However, this is not the case; instead, we use $J(\overline{u}) = 0$ and show the convergence of the functionals at the minimum values.

\begin{corollary}[Convergence of minimum values]\label{Cor: convOfMins}
   Let $0< s<1/3$, $1 < p <3/2$ be such that $\left(2/3 + s\right)p < 1$, and $\overline{u}_h \in X_h$ be the solution to problem \eqref{enhanced} for each fixed $h > 0$, and let $\overline{u}$ be the solution to problem \eqref{Eq: Mania}. Then it holds as $h\to 0^+$,  
    \begin{equation}\label{ConvMins}
     \lim_{h \rightarrow 0^+}  J^{\al}_h(\overline{u}_h) \ = \ J(\overline{u}).
    \end{equation}
\end{corollary}

\begin{proof}
By using $\{K_h\overline{u}\}_{h > 0}$ as the recovery sequence associated with $\overline{u}$, along with the minimality of $\{\overline{u}_h\}_{h > 0}$ and the non-negativity of  $\{J^{\al}_h(\cdot)\}_{h > 0}$, we have the inequality chain
    \begin{equation}\label{ConvOfMinsEq1}
        J(\overline{u}) \ \geq \ \underset{h\to 0^+}{ \mbox{\rm lim\,sup\ }}J^{\al}_h(K_h\overline{u}) \ \geq \ \underset{h\to 0^+}{ \mbox{\rm lim\,inf\ }}J^{\al}_h(\overline{u}_h) \ \geq \ 0 \ = \ J(\overline{u}),
    \end{equation}
    which completes the proof.
    Notice that the last equality in \eqref{ConvOfMinsEq1} is only true since $(s+2/3)p < 1$. 
\end{proof}


\section{$\Ga$-convergence for Foss's problem}\label{Sec: GaConvFoss}


The objective of this section, which is the second main goal of this paper, is to prove the $\Gamma$-convergence of the enhanced finite element discrete energy functional $F^\beta_h$ defined in \eqref{Eq: FossCutoff}.  To the end, we first state the two-dimensional analogue of Definition \ref{GaConvDef}.

\begin{definition}[$\Ga$-convergence]\label{Def: GaConvDefFoss}
Let $\alpha > 0$,  $0<s<1/3$ and $1 \leq p <3/2$ be such that $\left(2/3 + s\right)p < 1$, $Y_h$ denote the continuous piecewise linear finite element space \eqref{Eq: CtsPWLin2D}, and  $Y:=\mathcal{B} \cap W^{1 + s,p}(\Om)$. 
We say that the family of functionals $\{F^{\al}_h\}_{h>0}$, defined on $Y_h$,  {\em $\Ga$-converges} as $h\to 0^+$ to the functional $F$, defined on $Y$,  (written as $F^{\be}_h \xrightarrow{\Ga} F$) with respect to the strong  $W^{1, p}(\Om)$-topology if the following hold:
    \begin{enumerate} \label{GC1F}   
        \item[\rm (i)] \underline{The lim-inf property:} If $\{v_h\}_{h > 0} \subset Y_h$ 
        is such that $v_h \rightarrow v$ strongly in $W^{1, p}(\Om)$ then we have the {\em lim-inf inequality}
        \begin{equation}\label{Eq: liminfIneqCutoffFoss}
           \underset{h\to 0^+}{ \mbox{\rm lim\,inf}}\,  F^{\be}_h(v_h)  \geq F(v).
        \end{equation}\label{GC2F}
        \item[\rm (ii)] \underline{Recovery sequence property:} For any $v \in Y$, 
        there exists a {\em recovery sequence} $\{v_h\}_{h > 0}\subset Y_h$ such that $v_h \rightarrow v$ strongly in $W^{1, p}(\Om)$ and the following {\em lim-sup inequality} holds:
        \begin{equation}\label{Eq: limsupneqCutoffFoss}
         \underset{h\to 0^+}{ \mbox{\rm lim\,sup}}\, F^{\be}_h(v_h)   \leq F(v).
        \end{equation}
    \end{enumerate}
\end{definition}

The next two lemmas comprise an analog of Lemma \ref{Lem: ErrorEnergyIII}. Throughout this section, let $K_h$ be an operator that satisfies properties \eqref{property-1}--\eqref{property-3}; see Remark \ref{rem3.1}.

\begin{lemma}\label{Lem: tecRecSeqFossLem1}
    Let $1 \leq p < \infty$, $0 < s < 1$, and suppose $v \in \Bor \cap W^{1 + s, p}(\Om)$. Let $0 < \be < (1 + s)/14$ be fixed. Let $v_h:=K_hv$. Then there holds
    \begin{eqnarray}\label{Eq: tecRecSeqFossLem1}
        &&\int_{\Om}\Bigl(\frac{y}{y - 1}\Bigr)^{14}(|v_h|^{\frac{14 - 3y}{y - 1}} - |v|^{\frac{14 - 3y}{y - 1}})(|v_h|^{\frac{y}{y - 1}} - x)^2(\chi^{\be}_h(v_h)_x)^{14}\, dxdy \\
        &&\quad \lesssim h^{1 + s - 14\be}\max\bigl\{1, \|v\|^{22}_{L^{\infty}(\Om)} \bigr\}. \nonumber
    \end{eqnarray}
\end{lemma}

\begin{proof}
     This estimate follows readily from the definition of $\chi_h^\beta$, the $L^{\infty}$-stability of $K_h$ (see \eqref{property-1}), $\frac{3}{2} \leq y \leq \frac{5}{2}$, and the fact that $v \in L^{\infty}(\Om)$.
\end{proof}

\begin{lemma}\label{Lem: tecRecSeqFossLem2}
     Let $1 \leq p < \infty$, $0 < s < 1$, and suppose $v \in \Bor \cap W^{1 + s, p}(\Om)$. Let $0 < \be < (1 + s)/14$ be fixed, and let $v_h:=K_hv$. Then there holds
    \begin{eqnarray}\label{Eq: tecRecSeqFossLem2}
&&\int_{\Om}\Bigl(\frac{y}{y - 1}\Bigr)^{14}|v|^{\frac{14 - 3y}{y - 1}}[(|v_h|^{\frac{y}{y - 1}} - x)^2 - (|v|^{\frac{y}{y - 1}} - x)^2](\chi^{\be}_h(v_h)_x)^{14}dxdy \\
&&\quad \lesssim \ h^{1 + s - 14\be}\max\bigl\{1, \|v\|^{24}_{L^{\infty}(\Om)} \bigr\}.  \nonumber
    \end{eqnarray}
\end{lemma}

\begin{proof}
    Once again, we proceed by using the definition of $\chi_h^\be$, the $L^{\infty}$-stability of $K_h$ (see \eqref{property-1}), and the fact that $v \in L^{\infty}(\Om)$.
\end{proof}

This lemma comprises an analog of Lemma \ref{Lem: ErrorEnergyIV}. 

\begin{lemma}\label{Lem: tecRecSeqFossLem3}
     Let $1 \leq p < \infty$, $0 < s < 1$ and suppose $v \in Y:= \Bor \cap W^{1 + s, p}(\Om)$. Let $0 < \be < \frac{s}{13}$ be fixed, and let $v_h:=K_hv$. Then there holds
     \begin{eqnarray}\label{Eq: tecRecSeqFossLem3}
       && \int_{\Om}\Bigl(\frac{y}{y - 1}\Bigr)^{14}|v|^{\frac{14 - 3y}{y - 1}}(|v|^{\frac{y}{y - 1}} - x)^2|(\chi^{\be}_h(v_h)_x)^{14} - \chi^{\be}_h(v_x)^{14}|\, dxdy  \\
        &&\quad \lesssim h^{s - 13\be}\max\{1, \|v\|^{25}_{L^{\infty}(\Om)}\}.  \nonumber
     \end{eqnarray}
\end{lemma}

\begin{proof}
    This proof follows by using the $L^{\infty}$-stability of $K_h$ granted by \eqref{property-1} and a difference factorization.
\end{proof}

\begin{theorem}\label{Th: GaConvFoss}
     Let $0 < s < 1/3$ and $1 < p < \frac{3}{2}$ be fixed so that $(2/3 + s)p < 1$ . Assume that $\be < \min\{(1 + s)/14, s/13\}$ fixed , we have that $F^{\be}_h \xrightarrow{\Ga} F$ as $h\to 0^+$ in the strong $W^{1, p}(\Om)$-topology. 
\end{theorem}

\begin{proof}
To prove \eqref{Eq: liminfIneqCutoffFoss}, note that since \eqref{Eq: FossCutoff} is non-negative and convex in the gradient for each $h > 0$, \cite[Theorem 4.1]{schnake2017numerical} applies to obtain the lim-inf inequality. However, in this special setting, we may instead apply Fatou's Lemma straight away.

Now to prove \eqref{Eq: limsupneqCutoffFoss}, fix $v \in \Alg \cap W^{1 + s, 1}(\Om)$,  let  $v_h:=K_hv$ and $c_0:=66(13/14)^{14}$. Then we decompose $F^\beta_h(v_h)$ as
\begin{eqnarray}\label{Eq: FossDecomposition}
    \begin{aligned}
        F^{\be}_h(v_h) &:=c_0\int_{\Om}\Bigl(\frac{y}{y - 1}\Bigr)^{14}(|v_h|^{\frac{14 - 3y}{y - 1}} - |v|^{\frac{14 - 3y}{y - 1}})(|v_h|^{\frac{y}{y - 1}} - x)^2(\chi^{\al}_h(v_h)_x)^{14}\, dxdy  \\
        &\quad +c_0 \int_{\Om}\Bigl(\frac{y}{y - 1}\Bigr)^{14}|v|^{\frac{14 - 3y}{y - 1}}[(|v_h|^{\frac{y}{y - 1}} - x)^2 - (|v|^{\frac{y}{y - 1}} - x)^2](\chi^{\al}_h(v_h)_x)^{14}\,dxdy\\
        &\quad +c_0\int_{\Om}\Bigl(\frac{y}{y - 1}\Bigr)^{14}|v|^{\frac{14 - 3y}{y - 1}}(|v|^{\frac{y}{y - 1}} - x)^2[(\chi^{\al}_h(v_h)_x)^{14} - \chi^{\al}_h(v_x)^{14}\,]dxdy \\
        &\quad  + c_0\int_{\Om}\Bigl(\frac{y}{y - 1}\Bigr)^{14}|v|^{\frac{14 - 3y}{y - 1}}(|v|^{\frac{y}{y - 1}} - x)^2\chi^{\al}_h(v_x)^{14}\,dxdy.
    \end{aligned}
\end{eqnarray}
The first term on the right-hand side of \eqref{Eq: FossDecomposition} converges to $0$ due to Lemma \ref{Lem: tecRecSeqFossLem1}. The second term on the right-hand side of \eqref{Eq: FossDecomposition} converges to $0$ due to Lemma \ref{Lem: tecRecSeqFossLem2}. As for the third term, we bound it from above as
\begin{eqnarray}\label{Eq: FossDecomposition2}
     &&c_0\int_{\Om}\Bigl(\frac{y}{y - 1}\Bigr)^{14}|v|^{\frac{14 - 3y}{y - 1}}(|v|^{\frac{y}{y - 1}} - x)^2[(\chi^{\al}_h(v_h)_x)^{14} - \chi^{\al}_h(v_x)^{14}]dxdy \ \\
      &&\qquad \leq c_0 \int_{\Om}\Bigl(\frac{y}{y - 1}\Bigr)^{14}|v|^{\frac{14 - 3y}{y - 1}}(|v|^{\frac{y}{y - 1}} - x)^2|(\chi^{\al}_h(v_h)_x)^{14} - \chi^{\al}_h(v_x)^{14}|dxdy. \nonumber
\end{eqnarray}
Then the right-hand side of \eqref{Eq: FossDecomposition2} converges to $0$ due to Lemma \ref{Lem: tecRecSeqFossLem3}.

Finally, it is immediate that
\begin{equation}\label{Eq: FossDecomposition3}
    c_0\int_{\Om}\left(\frac{y}{y - 1}\right)^{14}|v|^{\frac{14 - 3y}{y - 1}}(|v|^{\frac{y}{y - 1}} - x)^2\chi^{\al}_h(v_x)^{14}dxdy \ \leq \ F(v).
\end{equation}
The proof is completed by combining \eqref{Eq: FossDecomposition}--\eqref{Eq: FossDecomposition3} and sending $h \rightarrow 0^+$.
\end{proof}

Similarly to Mani\`a's problem, we can also show a convergence result for minimum values analogous to Corollary \ref{Cor: convOfMins}.

\begin{corollary}[Convergence of minimum values]\label{Cor: ConvOfMinsFoss}
   Let $0< s<1/3$, $1 \leq p <3/2$ be such that $\left(2/3 + s\right)p < 1$, and $\overline{v}_h \in Y_h$ be the solution to the problem \eqref{Eq: enhancedFoss}, and let $\overline{v}$ be the minimizer of Foss's functional defined in \eqref{Eq: Foss}. Then there holds as $h\to 0^+$,  
    \begin{equation}\label{Eq: ConvMinsFoss}
     \lim_{h \rightarrow 0^+}  F^{\be}_h(\overline{v}_h) \ = \ F(\overline{v}).
    \end{equation}
\end{corollary}

We note that the thresholds on $s$ and $p$ required for convergence are the same as for Mani\`a's problem.

\section{Numerical experiments}\label{Sec: Numerics}

In this section,  we provide numerical experiments to verify the convergence \eqref{ConvMins} and \eqref{Eq: ConvMinsFoss}. For Mani\`a's problem (resp. Foss's problem), let $\widetilde{u}_h := \argmin_{u_h \in X_h}J(u_h)$ (let $\widetilde{v}_h := \argmin_{v_h \in Y_h}F(v_h)$), and the tables will numerically demonstrate the LGP. We also compute values of $J^{\al}_h(I_h\overline{u})$ (resp. $F^{\al}_h(I_h\overline{v})$) to demonstrate the lim-sup inequality \eqref{limsupneqCutoff} (resp. \eqref{Eq: limsupneqCutoffFoss}), where $I_h$ denotes the nodal interpolant. Finally, we show the rates of convergence between the cutoff functional and the limiting functional, as well as the $L^2$- and $W^{1, 1}$-errors between the discrete and continuous minimizers. Tables \ref{table: alphaNudgeAl = 1/4}-\ref{table: alphaNudgeAlRates = 2/7} pertain to Mani\`a's problem, while Tables \ref{table: betaNudgeAlFoss = 1/4}-\ref{table: alphaNudgeBeRates = 2/7} pertain to Foss's problem. The lack of consistent convergence rates for the errors (in both norms) can be attributed to the low regularity of minimizers for $J$ and $F$, along with the lack of coercivity of these problems.

Now we provide a few details on the implementation. In all cases, the built-in MATLAB function {\em fminunc} was used to perform the optimizations. For Mani\`a's problem, the initial guess for the optimization is $u_0(x) = x$, while for Foss's problem, it is $u_0(x, y) = x$. Additionally, for Foss's problem, the parallel processing functionality \textit{parpool} was utilized with four workers, and we performed the optimization routine in two steps: first compute $\widetilde{u}_h$ and then use this as an initial condition to compute $\overline{u}_h$.

\begin{table}
\centering
 \begin{tabular}{||l l l l l l||}
 \hline 
$h$  & $1/10$ & $1/20$  & $1/40$  & $1/80$  & $1/160$   \\ [0.3cm]
 \hline\hline 
$J(\widetilde{u}_h)$ & $0.0374372$  & $0.0325811$  & $0.0297002$ & $0.0279371$ & $ 0.0268322$  \\
 \hline
$J^{\al}_h(\overline{u}_h)$ & $0.0016898$ & $0.0006027$ & $0.0002222$ & $0.00008598$  &  $0.0000331$ \\
 \hline
 $J^{\al}_h(I_h\overline{u})$ & $0.002415$ & $0.0008631$ & $0.0003091$ & $0.00011007$ & $0.0000391$ \\
 \hline
\end{tabular}
\caption{\label{table: alphaNudgeAl = 1/4} Values for Mani\`a functional and cutoff when $\al = 1/4$}
\end{table}

\begin{table}
\centering
 \begin{tabular}{||l l l l l l||}
 \hline 
$h$  & $1/10$ & $1/20$  & $1/40$  & $1/80$  & $1/160$  \\ [0.3cm]
 \hline\hline 
$J(\widetilde{u}_h)$ & $0.0374372$  & $0.0325811$  & $0.0297002$ & $0.0279371$ & $ 0.0268322$ \\
 \hline
$J^{\al}_h(\overline{u}_h)$ & $0.0374372$ & $0.0325811$ & $0.000477$  & $0.0002092$ & $0.0000937$ \\
 \hline
 $J^{\al}_h(I_h\overline{u})$ & $0.0039519$ & $0.0016299$& $0.0006802$ & $0.0002807$ & $0.0001157$\\
 \hline
\end{tabular}
\caption{\label{table: alphaNudgeAl = 2/5} Values for Mani\`a functional and cutoff when $\al = 2/7$ }
\end{table}

\begin{table}
\centering
 \begin{tabular}{||l l l l l l||}
 \hline 
$h$  & $1/10$ & $1/20$  & $1/40$  & $1/80$  & $1/160$   \\ [0.3cm]
 \hline\hline 
$|J^{\al}_h(\overline{u}_h) - J(\overline{u})|$ & $0.0016898$ & $0.0006027$ & $0.0002222$ & $0.00008598$  &  $0.00003314$\\
 \hline
Rate 
& $-$ & $1.48782$  & $1.43963$& $1.36974$  &  $1.37540$ \\
 \hline
 $\|\overline{u}_h - \overline{u}\|_{L^2}$ & $0.0364538$ & $0.0204453$ & $0.0114443$ & $0.00640736$ & $0.00382858$\\
 \hline
 Rate 
 & $-$& $0.834300$ & $0.837140$ & $0.836827$ & $0.742921$ \\
 \hline
 $\|\overline{u}_h - \overline{u}\|_{W^{1, 1}}$ & $0.0248857$ & $0.0103673$ & $0.004696$& $0.00227012$ & $0.00070888$\\
 \hline
 Rate 
 & $-$& $1.26328$  & $1.14253$& $1.04867$ & $1.67916$ \\
 \hline
\end{tabular}
\caption{\label{table: alphaNudgeAlRates = 1/4} Error measurements for Mani\`a functional and cutoff when $\al = 1/4$ }
\end{table}

\begin{table}
\centering
 \begin{tabular}{||l l l l l l||}
 \hline 
$h$  & $1/10$ & $1/20$  & $1/40$  & $1/80$  & $1/160$   \\ [0.3cm]
 \hline\hline 
$|J^{\al}_h(\overline{u}_h) - J(\overline{u})|$ & $0.0374372$ & $0.032581$ & $0.000477$  & $0.0002092$ & $0.00009367$ \\
 \hline
Rate 
& $-$  & $0.200437$  & $6.09387$ & $1.18923$ & $1.15915$  \\
 \hline
 $\|\overline{u}_h - \overline{u}\|_{L^2}$ & $0.215934$ & $0.207333$ & $0.0114628$& $0.0064123$ & $0.00360301$ \\
 \hline
 Rate 
 & $-$& $0.058641$& $4.17692$ & $0.838046$ & $0.831639$ \\
 \hline
 $\|\overline{u}_h - \overline{u}\|_{W^{1, 1}}$ & $0.256652$ & $0.219786$ & $0.0045039$& $0.0021616$ & $0.00110797$\\
 \hline
 Rate 
 & $-$& $0.223714$ & $5.60880$ & $1.05907$ & $0.964174$ \\
 \hline
\end{tabular}
\caption{\label{table: alphaNudgeAlRates = 2/7} Error measurements for Mani\`a functional and cutoff when $\al = 2/7$}
\end{table}

\begin{table}
\centering
 \begin{tabular}{||l l l l||}
 \hline 
$h$  & $1/6$ & $1/12$  & $1/24$  \\ [0.3cm]
 \hline\hline 
$F(\widetilde{v}_h)$ & $14.4843$  & $5.71926$ & $3.21556$  \\
 \hline
$F^{\be}_h(\overline{v}_h)$ & $3.74796$ & $0.190786$ &  $0.0121433$  \\
 \hline
 $F^{\be}_h(I_h\overline{v})$ & $2.17234$ & $0.298273$ & $0.0399395$ \\
 \hline
\end{tabular}
\caption{\label{table: betaNudgeAlFoss = 1/4}Values for Foss functional and cutoff when $\be = 1/4$ }
\end{table}

\begin{table}
\centering
 \begin{tabular}{||l l l l||}
 \hline 
$h$   & $1/6$ & $1/12$  & $1/24$  \\ [0.3cm]
 \hline\hline 
$F(\widetilde{v}_h)$ & $14.4843$  & $5.71926$ & $3.21556$   \\
 \hline
$F^{\be}_h(\overline{v}_h)$ & $7.08834$& $1.53267$  & $0.527008$ \\
 \hline
 $F^{\be}_h(I_h\overline{v})$ & $3.86372$ & $0.762079$ & $0.149960$ \\
 \hline
\end{tabular}
\caption{\label{table: betaNudgeAlFoss = 2/7} Values for Foss functional and cutoff when $\be = 2/7$ }
\end{table}

\begin{table}
\centering
 \begin{tabular}{||l l l l||}
 \hline 
$h$  & $1/6$ & $1/12$ & $1/24$    \\ [0.3cm]
 \hline\hline 
$|F^{\be}_h(\overline{v}_h) - F(\overline{v})|$ & $3.74796$ & $0.190786$ &  $0.0121433$  \\
 \hline
Rate 
& $-$  & $4.29608$ & $3.97372$   \\
 \hline
 $\|\overline{v}_h - \overline{v}\|_{L^2}$ & $0.233337$ & $0.112987$ & $0.0567265$  \\
 \hline
 Rate 
 & $-$& $1.04626$& $0.994062$  \\
 \hline
  $\|\overline{v}_h - \overline{v}\|_{W^{1, 1}}$ &  $0.429554$ & $0.327292$& $0.286032$ \\
 \hline
 Rate 
 & $-$ & $0.392261$& $0.194402$  \\
 \hline
\end{tabular}
\caption{\label{table: alphaNudgeBeRates = 1/4} Error measurements for Foss functional and cutoff when $\be = 1/4$}
\end{table}

\begin{table}
\centering
 \begin{tabular}{||l l l l||}
 \hline 
$h$  & $1/6$ & $1/12$ & $1/24$   \\ [0.3cm]
 \hline\hline 
$|F^{\be}_h(\overline{v}_h) - F(\overline{v})|$ & $7.08834$ & $1.53266$ & $0.527008$  \\
 \hline
Rate 
& $-$  & $2.20941$  & $1.54014$  \\
 \hline
 $\|\overline{v}_h - \overline{v}\|_{L^2}$ & $0.216776$ & $0.110129$  & $0.0582341$ \\
 \hline
 Rate 
 & $-$ & $0.977011$ & $0.919258$ \\
 \hline
 $\|\overline{v}_h - \overline{v}\|_{W^{1, 1}}$ & $0.429245$ & $0.332017$ & $0.291064$ \\
 \hline
 Rate 
 & $-$ & $0.370544$ & $0.189921$ \\
 \hline
\end{tabular}
\caption{\label{table: alphaNudgeBeRates = 2/7} Error measurements for Foss functional and cutoff when $\be = 2/7$}
\end{table}

In the results of Tables \ref{table: alphaNudgeAl = 1/4}-\ref{table: alphaNudgeAlRates = 2/7}, if there was no LGP,  then $J(\cdot)$ and $J^{\al}_h(\cdot)$ would have similar convergence behaviors for all values of $\al$. Therefore, we opt to tune $\al$ in our numerical results (and similarly for $\be$ in the 2D problem; see Tables \ref{table: betaNudgeAlFoss = 1/4}-\ref{table: alphaNudgeBeRates = 2/7}). Moreover, the numerical method does not require a priori knowledge about the minimizer; it only requires that the minimizer has a slightly higher differentiability than that inferred by the space $W^{1, 1}(0, 1)$. Finally, we note that additional numerical results for both problems are given in \cite{feng2016enhanced, schnake2017numerical}, including demonstrations that using higher-order finite element methods does not improve the convergence properties of these problems.





\section{Conclusion}\label{Sec: future}

In this paper, we presented two enhanced finite element methods for two calculus of variations problems, namely,  Mani\`a's problem and Foss's problem, both of which suffer from the Lavrentiev Gap Phenomenon (LGP),  and established complete $\Gamma$-convergence proofs for both methods. 
To construct the recovery sequence, we used an abstract interpolation-like numerical operator $K_h$ that needs to satisfy a criterion consisting of certain stability and approximation properties.  The convergence proofs may provide a blueprint or framework for developing numerical methods and establishing the $\Ga$-convergence for other calculus of variations problems with the LGP.  

A natural next step is to generalize the described framework to other problems with the Lavrentiev Gap in higher dimensions, such as the elasticity functional described in \cite{almi2024new}, and the two-dimensional version of the Mani\`a functional described in \cite{caetano2004example}. One can also try juxtaposing the cutoff functional with non-standard finite element methods, such as Discontinuous Galerkin (DG) methods.

Finally, there is the possibility of studying nonlocal analogues of the local functionals with the LGP by choosing these functionals to converge to the corresponding local functional in the sense of Brezis-Bourgain-Mironescu \cite{bourgain2001another}. We conjecture that using such constructions has the ability to remove the Lavrentiev Gap, enabling the use of standard finite elements for these types of problems. 


\section*{Acknowledgements}

The work of the first author was partially supported by the NSF grant DMS-2309626.  
Both authors would like to thank Professor Abner Salgado 
 for providing valuable comments on earlier drafts of this paper.



\appendix 
\section*{} \label{appendixA}
We provide a proof of the fractional-order differentiability for the minimizer $\overline{v}$ of  Foss's problem. We note that the condition $sp < 1$ used earlier in the paper is less stringent than the conditions required for this result.
 
\begin{proposition} \label{Lem: FracMin2-x}
	The function $\overline{v}(x, y) := x^{\frac{y - 1}{y}}$ belongs to $W^{1 + s, p}(\Om)$ for $1 \leq p < \frac{3}{2}$ and $s < \frac{1}{p} - \frac{2}{3}$.
\end{proposition}

\begin{proof}
	Fix $0 \leq s < 1$ and $1 \leq p < \infty$ so that $s(p + 2/3) < 1$,  it suffices to show that the partial derivatives $\overline{v}_x$ and $\overline{v}_y$  of $\overline{v}$ satisfy the following four claims:
	\[
		(a)\,  \overline{v}_x \in L^p(\Om),\quad 
		 (b)\, \overline{v}_x \in W^{s, p}(\Om), \quad
		 (c)\,  \overline{v}_y \in L^p(\Om), \quad 
		 (d)\,  \overline{v}_y \in W^{1, p}(\Om). 
	\]
	
	To prove (a),  we simply estimate
	\begin{equation}\label{Eq: FracMinUxEq1-x}
		\|\overline{v}_x\|^p_{L^p(\Om)} \ \lesssim \ \int_{\Om}|x^{-\frac{1}{y}}|^pdxdy \ < \ \infty.
	\end{equation}
	
	To prove (b), we use Fubini's Theorem and a change of variables to obtain the estimate 
	\begin{eqnarray}\label{Eq: FracMinS2Eq1-y}
		\begin{aligned}
			&\iint_{\Om \times \Om}\frac{|\overline{v}_x(x, y) - \overline{v}_x(\widetilde{x}, y)|^p}{|(x, y) - (\widetilde{x}, \widetilde{y})|^{2 + sp}}d(x, y)d(\widetilde{x}, \widetilde{y}) \\
			&\qquad  \lesssim \int^1_0\int^{5/2}_{3/2}\int^1_0\frac{|x^{-1/y} - \widetilde{x}^{-1/y}|^p}{|x - \widetilde{x}|^{1 + sp}}d\widetilde{x}dydx \\
			&\qquad  \leq\int^1_0\int^1_0\frac{|x^{-2/3} - \widetilde{x}^{-2/3}|^p}{|x - \widetilde{x}|^{1 + sp}}d\widetilde{x}dx \ < \ \infty,
		\end{aligned}
	\end{eqnarray}
	where we have used that $x^{-2/3} \in W^{s, p}(0, 1)$ when $s < \frac{1}{p} - \frac{2}{3}$. We also require a second estimate, following from the Mean-Value Theorem: 
	\begin{eqnarray}\label{Eq: FracMinS2Eq2-2}
		\begin{aligned}
			&\iint_{\Om \times \Om}\frac{|\overline{v}_x(\widetilde{x}, y) - \overline{v}_x(\widetilde{x}, \widetilde{y})|^p}{|(x, y) - (\widetilde{x}, \widetilde{y})|^{2 + sp}}d(x, y)d(\widetilde{x}, \widetilde{y}) \\
			&\qquad \lesssim \iint_{\Om \times \Om}\frac{|\ln(\widetilde{x})|^p\widetilde{x}^{-2p/3}|y - \widetilde{y}|^p} {|(x, y) - (\widetilde{x}, \widetilde{y})|^{2 + sp}}d(x, y)d(\widetilde{x}, \widetilde{y}) \\
			&\qquad  \leq \iint_{\Om \times \Om}\frac{|\ln(\widetilde{x})|^p\widetilde{x}^{-2p/3}|y - \widetilde{y}|^p}{|y - \widetilde{y}|^{2 + sp}}d(x, y)d(\widetilde{x}, \widetilde{y}) \\
			&\qquad  = \left(\int^1_0|\ln(\widetilde{x})|^p\widetilde{x}^{-2p/3}dx\right)\left(\int^{5/2}_{3/2}\int^{5/2}_{3/2}\frac{1}{|y - \widetilde{y}|^{2 + (s - 1)p}}d\widetilde{y}dy\right) \ < \ \infty.
		\end{aligned}
	\end{eqnarray}
	Finally, combining \eqref{Eq: FracMinS2Eq1-y} and \eqref{Eq: FracMinS2Eq2-2}, we  obtain
	\begin{eqnarray}\label{Eq: FracMinS2Eq3-z}
		\begin{aligned}
			[\overline{v}_x]^p_{W^{s, p}(\Om)}  &\lesssim  \iint_{\Om \times \Om}\frac{|\overline{v}_x(x, y) - \overline{v}_x(\widetilde{x}, y)|^p}{|(x, y) - (\widetilde{x}, \widetilde{y})|^{2 + sp}}d(x, y)d(\widetilde{x}, \widetilde{y})  \\ 
			&\qquad +\iint_{\Om \times \Om}\frac{|\overline{v}_x(\widetilde{x}, y) - \overline{v}_x(\widetilde{x}, \widetilde{y})|^p}{|(x, y) - (\widetilde{x}, \widetilde{y})|^{2 + sp}}d(x, y)d(\widetilde{x}, \widetilde{y}) \ < \ \infty.
		\end{aligned}
	\end{eqnarray}
	Thus, (b) holds.
	
	Assertion (c) follows from the integrability of any positive power of the natural logarithm function in 1D and $x^{\frac{y-1}{y}} \in L^{\infty}(\Om)$.
	
	Finally,  assertion (d) immediately follows from the following calculations:
	\[
	\overline{v}_{yx}(x, y) = \frac{2x^{-\frac{1}{y}}y - \ln(x)x^{-\frac{1}{y}}}{y^3},
	\qquad \overline{v}_{yy}(x, y) = \frac{\ln(x)^2x^{\frac{y - 1}{y}}}{y^4} - \frac{2\ln(x)x^{\frac{y - 1}{y}}}{y^3}.
	\]
	The proof is complete. 
\end{proof}

\begin{remark}\label{Rmk: SameThreshold-x}
	The fractional differentiability and integrability thresholds for the minimizers of \eqref{Eq: Mania} and \eqref{Eq: Foss} are the same, apart from the number of dimensions. This further explains why the convergence analysis for the two problems is similar.
\end{remark}

\end{document}